\documentclass{article}
\usepackage[utf8]{inputenc}
\usepackage{amsmath,amssymb,amsthm}
\usepackage{graphicx}
\usepackage[colorlinks=true,citecolor=black,linkcolor=black,urlcolor=blue]{hyperref}
\usepackage{url}
\usepackage{cleveref}
\usepackage{subcaption}
\usepackage{todonotes}
\usepackage{enumerate}

\newcommand{\arxiv}[1]{\href{http://arxiv.org/abs/#1}{\texttt{arXiv:#1}}}

\def\acc#1{\left\{ #1 \right\}}

\renewcommand{\le}{\leqslant}
\renewcommand{\ge}{\geqslant}

\theoremstyle{plain}
\newtheorem{theorem}{Theorem}
\newtheorem{lemma}[theorem]{Lemma}
\newtheorem{corollary}[theorem]{Corollary}

\theoremstyle{definition}

\theoremstyle{remark}

\title{More characterizations of morphic words}
\author{
Golnaz Badkobeh\footnote{Department of Computer Science, City University of London, United Kingdom.\\
e-mail \href{mailto:Golnaz.Badkobeh@city.ac.uk}{\texttt Golnaz.Badkobeh@city.ac.uk}.}
\and
Pascal Ochem\footnote{LIRMM, CNRS, Universit\'e de Montpellier, France.\\
e-mail \href{mailto:ochem@lirmm.fr}{\texttt ochem@lirmm.fr}.}}
\date{2023}

\begin{document}

\maketitle

\begin{abstract}
An interesting phenomenon in combinatorics on words is when every recurrent word satisfying
some avoidance constraints has the same factor set as a morphic word.
An early example is the Hall-Thue word, fixed point of the morphism
$\texttt{0}\to\texttt{012}$, $\texttt{1}\to\texttt{02}$, $\texttt{2}\to\texttt{1}$,
which is essentially the only ternary word avoiding squares and the words \texttt{010} and \texttt{212}.
We provide some examples of this phenomenon from various contexts.
\end{abstract}

\section{Introduction}\label{sec:intro}
Many recent papers in combinatorics on words consider some factorial languages of a given type (described below) and
determine which ones are infinite.
\begin{itemize}
\item Binary words with few antisquares and low critical exponent~\cite{BCMORS2023}.
\item Binary words with few pairs of complementary factors and low critical exponent~\cite{CDOORS2023}.
\item Binary words with few distinct palindromes and low critical exponent~\cite{DOO2023}.
\item Binary words avoiding formulas with two variables~\cite{OchemRosenfeld2016}.
\end{itemize}
In each paper, most languages contain exponentially many words (with respect to their length)
and a few languages contain polynomially many words. The polynomial cases are both more interesting
and more difficult to handle. In this paper, instead of considering exhaustively the
exponential and polynomial cases of one avoidance framework, we focus on the polynomial cases
that we encountered in various context. This widens the scope of our previous paper~\cite{BO15}
which considered polynomial cases of words with few squares.

\subsection{Notation}
Let $\Sigma_k$ denote the $k$-letter alphabet $\acc{\texttt{0,1},\dots,\texttt{k-1}}$.
Let $\varepsilon$ denote the empty word.
A finite word is \emph{recurrent} in an infinite word $w$ if it appears as a factor of $w$ infinitely many times.
An infinite word $w$ is \emph{recurrent} if all its finite factors are recurrent in $w$.
If a morphism $f$ is such that $f(\texttt{0})$ starts with $\texttt{0}$, then the \emph{fixed point} of $f$
is the unique word $w=f^\omega(\mathtt{0})$ starting with $\texttt{0}$ and satisfying $w=f(w)$. 
An infinite word is \emph{pure morphic} if it is the fixed point of a morphism.
An infinite word is \emph{morphic} if it is the image $g(f^\omega(\texttt{0}))$ by a morphism $g$ of a pure morphic word $f^\omega(\texttt{0})$.
% The \emph{factor complexity} of an infinite word or a language is the number of factors of length $n$ of the infinite word or the language.
A pattern $P$ is a finite word of variables over the alphabet $\acc{A,B,\dots}$.
We say that a variable is \emph{isolated} if it appears exactly once in the pattern.
Following Cassaigne~\cite{Cassaigne1994}, we associate to a pattern $P$ the \emph{formula} $f$
obtained by replacing every isolated variable in $p$ by a dot.
For example, the formula associated to the pattern $ABBACABADAA$ is $ABBA.ABA.AA$.
The factors between the dots are called \emph{fragments}.
An \emph{occurrence} of a formula $f$ in a word $w$ is a non-erasing morphism $h:\Delta^*\to\Sigma^*$
such that the $h$-image of every fragment of $f$ is a factor of $w$.
A word $w$ (finite or infinite) \emph{avoids} a formula $f$ if it contains no occurrence of $f$.

A \emph{repetition} is a factor of the form $r=u^nv$, where $u$ is non-empty and $v$ is a prefix of $u$.
Then $|u|$ is the \emph{period} of the repetition $r$ and its \emph{exponent} is $|r|/|u|$. 
A \emph{square} is a repetition of exponent 2. Equivalently, it is an occurrence of the pattern $AA$. 
An \emph{overlap} is a repetition with exponent strictly greater than 2.
We use the notation $SQ_t$ for the pattern corresponding to squares with period at least $t$,
that is, $SQ_1=AA$, $SQ_2=ABAB$, $SQ_3=ABCABC$, and so on.
A morphism $m$ is given in the format $m(\texttt{0})/m(\texttt{1})/...$

\subsection{Avoiding words}
A popular theme in combinatorics on words is to construct or at least prove the existence of
an infinite word over $\Sigma_k$ that satisfies some given avoidance constraints, such as
limiting or forbidding repetitions and/or patterns.
A step further is to describe the avoiding words over $\Sigma_k$. The following situations occur in the literature.
\begin{enumerate}
 \item There are exponentially many avoiding words (with respect to their length).
 Classical examples are square-free words over $\Sigma_3$
 and words over $\Sigma_2$ that contain only 3 distinct squares~\cite{Shur:2012,Ochem2004}.
 \item Every recurrent avoiding word has the same factor set as one morphic word.
 For example, Thue~\cite{Thue06,Thue:1912} proved that every bi-infinite word
 avoiding squares and the words \texttt{010} and \texttt{212} has the same factor set as 
 the fixed point of the morphism $\texttt{012}/\texttt{02}/\texttt{1}$.
 %$\texttt{0}\to\texttt{012}$, $\texttt{1}\to\texttt{02}$, $\texttt{2}\to\texttt{1}$.
 \item There is a finite set $S$ of morphic words such that every recurrent avoiding word has the same factor set as one word in $S$.
For example, every bi-infinite binary word avoiding the formula $AA.ABA.ABBA$
 has the same factor set as one of 4 specific binary morphic words~\cite{OchemRosenfeld2016}.
 \item There are polynomially many avoiding words, but there are infinitely many recurrent avoiding words
 such that no two of them have the same factor set. The only known examples of this case are the binary words
 avoiding $ABACA.ABCA$ and the binary words avoiding $ABAC.BACA.ABCA$~\cite{OchemRosenfeld2017}.
\end{enumerate}
We do not know whether any other situation is possible.

In this paper, we identify quite natural examples of the second situation, with a unique morphic word.
Let $w$ be a morphic word over $\Sigma_k$ and $\mathcal{C}$ be a set of avoidance constraints.
We say that $\mathcal{C}$ \emph{characterizes} $w$
if every recurrent word over $\Sigma_k$ avoiding $\mathcal{C}$ has the same factor set as $w$.
% Equivalently, this means that $w$ avoids $\mathcal{C}$ and that for every finite factor $f$ of $w$,
% there are only finitely many words over $\Sigma_k$ avoiding $\mathcal{C}$ and $f$. 
Every morphic word can be written as $g(f^\omega(\texttt{0}))$ in many different ways.
In particular, $g$ can be chosen to be a coding (i.e., $1$-uniform).
For our morphic words that admit a characterization, the defining morphisms are found
such that the fixed point $f^\omega(\texttt{0})$ is a famous pure morphic word
which, unsurprisingly, also admits a characterization.
Let us present these useful pure morphic words.

\begin{itemize}
\item ${\bf f}$ is the Fibonacci word, fixed point of $\texttt{01}/\texttt{0}$.\\
It is characterized by $\acc{AAAA,\texttt{11},\texttt{000},\texttt{10101}}$~\cite{BCMORS2023}.\\
In~\cite{BCMORS2023}, it is shown that the image of ${\bf f}$ by $\texttt{01}/\texttt{11}$
is the binary word containing no antisquare other than \texttt{01} and \texttt{10} with the least
critical exponent.
\item $b_3$ is the Hall-Thue word, the fixed point of $\texttt{012}/\texttt{02}/\texttt{1}$.\\
It is characterized by $\acc{AA,\texttt{010},\texttt{212}}$~\cite{Thue06,Thue:1912}.\\
Some images of $b_3$ are characterized as containing only a few squares and overlaps~\cite{BO15},
or as avoiding formulas with two variables~\cite{OchemRosenfeld2016}.
Other interesting images of $b_3$ are described in~\Cref{sec:b3}.
\item ${\bf p}$ is the fixed point of $\texttt{01}/\texttt{21}/\texttt{0}$.\\
It is characterized by $\{AAA,\texttt{00},\texttt{11},\texttt{22},\texttt{20},\texttt{212},\texttt{0101},\texttt{02102},\texttt{121012},$\\
$\texttt{01021010},\texttt{21021012102}\}$~\cite{CORS2022}.\\
Some images of ${\bf p}$ are characterized as minmizing the critical exponent of binary words containing few pairs
of complementary factors~\cite{CDOORS2023} or few palindromes~\cite{DOO2023}.
\item $b_5$ is the fixed point of $\texttt{01}/\texttt{23}/\texttt{4}/\texttt{21}/\texttt{0}$.\\
It is characterized by $\{AA,\texttt{02},\texttt{03},\texttt{13},\texttt{14},\texttt{20},\texttt{24},\texttt{31},\texttt{32},\texttt{40},\texttt{41},\texttt{43},\texttt{121},$\\
$\texttt{212},\texttt{304},\texttt{3423},\texttt{4234}\}$~\cite{BO15}.\\
The other two square-free words of Thue, namely those avoiding $\acc{\texttt{010},\texttt{020}}$ and $\acc{\texttt{121},\texttt{212}}$,
are images of $b_5$~\cite{BO15}. Other interesting images of $b_5$ are described in~\Cref{sec:b5}.
\item ${\bf pd}$ is the period-doubling word, the fixed point of $\texttt{01}/\texttt{00}$.
James Currie~\cite{Currie2023} has established that ${\bf pd}$ is characterized by $\acc{AAAA, AAABABAA,\texttt{11},\texttt{1001}}$.
We show in~\Cref{sec:pd} that ${\bf pd}$ is also characterized by $\acc{AA.ABAB.BB,\texttt{11}}$.\\
\end{itemize}

Throughout the paper, we informally say that we have checked that some morphic word $w$ avoids a pattern $P$.
This means that we have successfully run the program written by Matthieu Rosenfeld that implements the algorithm of
Julien Cassaigne~\cite{cassaignealgo} to test whether a morphic word avoids a pattern, with $w$ and $P$ as input.

\section{Images of $b_3$}\label{sec:b3}
We start with a useful lemma using ideas from~\cite{BO15}.
\begin{lemma}\label{lem:b3}
Let $w$ be a bi-infinite word and let $m:\Sigma^*_3\rightarrow\Sigma^*_k$ be such that 
\begin{itemize}
 %\item $w$ has the same factor set of length $5\times\max\acc{|m(\textnormal{\texttt{0}})|,|m(\textnormal{\texttt{1}})|,|m(\textnormal{\texttt{2}})|}$ as $m(b_3)$,
 \item $w$ is in $\acc{m(\textnormal{\texttt{0}}),m(\textnormal{\texttt{1}}),m(\textnormal{\texttt{2}})}^\omega$,
 \item $m(\textnormal{\texttt{0}})=axb$ and $m(\textnormal{\texttt{1}})=ab$,
 \item $w$ avoids $SQ_t$ with $t\le\min\acc{|m(\textnormal{\texttt{0}})|,|m(\textnormal{\texttt{1}})|,|m(\textnormal{\texttt{2}})|}$.
\end{itemize}
then $w$ has the same factor set as $m(b_3)$.
\end{lemma}

\begin{proof}
Since $w\in\acc{m(\texttt{0}),m(\texttt{1}),m(\texttt{2})}^\omega$, we write $w=m(r_3)$,
where $r_3$ is a bi-infinite ternary word.
Notice that $r_3$ is square-free, since otherwise $m(r_3)$ would contain $SQ_t$.
Moreover, $r_3$ avoids words of the form $\texttt{0}u\texttt{1}u\texttt{0}$, with $u\in\Sigma_3^*$, since
$m(\texttt{0}u\texttt{1}u\texttt{0})=axbm(u)abm(u)axb$ contains $(bm(u)a)^2$, which is an occurrence of $SQ_t$.
Now, by Lemma 3.2 in~\cite{aabbc}, a bi-infinite ternary square-free word avoids factors of the form
$\texttt{0}u\texttt{1}u\texttt{0}$ with $u\in\Sigma_3^*$ if and only if it avoids $\acc{\texttt{010},\texttt{212}}$.
So $r_3$ has the same factor set as $b_3$.
By taking the $m$-image, $w$ has the same factor set as $m(b_3)$.
\end{proof}

\subsection{Four squares}\label{ssec:4sq}
Let $g_4$ be the morphism from $\Sigma_3^*$ to $\Sigma_2^*$ defined by
$$\begin{array}{l}
 g_4(\texttt{0})=\texttt{00010011000111011},\\
 g_4(\texttt{1})=\texttt{000100111011},\\
 g_4(\texttt{2})=\texttt{00111}.\\
\end{array}$$

\begin{theorem}
Every binary bi-infinite word containing only the squares \textnormal{\texttt{00}}, \textnormal{\texttt{11}},
\textnormal{\texttt{001001}}, and \textnormal{\texttt{110110}} has the same factor set as $g_4(b_3)$.
\end{theorem}

\begin{proof}
We say that a binary word, finite or infinite, is \emph{good} if it contains no squares besides 
\texttt{00}, \texttt{11}, \texttt{001001}, or \texttt{110110}.
It is easy to check that a bi-infinite word is good if and only if it
avoids $SQ_4$ and $\acc{\texttt{0000},\texttt{1111},\texttt{0101},\texttt{1010},\texttt{10010},\texttt{01101}}$.
We have checked that $g_4(b_3)$ is good.

We construct the set $S_4^{100}$ defined as follows:
a word $v$ is in $S_4^{100}$ if and only if there exists a good word $pvs$ such that $|p|=|v|=|s|=100$.
We check that $S_4^{100}$ is exactly the set of factors of length 100 of $g_4(b_3)$.
Thus, if $r_2$ is any bi-infinite good word, its factors of length 100
are also factors of $g_4(b_3)$.
Now, every element of $S_4^{100}$ contains the factor $g_4(\texttt{0})$
and every element of $S_4^{100}$ with prefix $g_4(\texttt{0})$ has a prefix in $\acc{g_4(\texttt{0120}),g_4(\texttt{01210}),g_4(\texttt{020}),g_4(\texttt{0210})}$.
So $r_2$ is in $\acc{g_4(\texttt{012}),g_4(\texttt{0121}),g_4(\texttt{02}),g_4(\texttt{021})}^\omega$
and thus $r_2$ is in $\acc{g_4(\texttt{0}),g_4(\texttt{1}),g_4(\texttt{2})}^\omega$.
Then $r_2$ has the same factor set as $g_4(b_3)$ by~\Cref{lem:b3}.
\end{proof}
% echo ABCDABCD 3 012 02 1 00010011000111011 000100111011 00111 | ./dn
% Size of border is 484
% 
% The formula ABCDABCD is avoided by the morphism !!!!!!!!!!!!!!!!!!
% The set of patterns is of size 24360

\begin{corollary}
There exists an infinite $3^+$-free binary word with two overlaps and four squares, avoiding squares of even periods.
\end{corollary}

Notice that Theorem 4 in \cite{BC-IPL} gives the construction of infinite words containing
only the four squares \texttt{00}, \texttt{11}, \texttt{001001}, \texttt{100100} and the two cubes
\texttt{000}, \texttt{111}. However, they contain a third overlap: \texttt{1001001}.

% echo ABABA 3 012 02 1 00010011000111011 000100111011 00111 | ./dn
% Size of border is 484
% 
% The formula ABABA is avoided by the morphism !!!!!!!!!!!!!!!!!!
% The set of patterns is of size 631

\subsection{Five squares}\label{ssec:5sq}
Let $g_5$ be the morphism from $\Sigma_3^*$ to $\Sigma_2^*$ defined by

$$\begin{array}{l}
 g_5(\texttt{0})=\texttt{0000100000111000011000111},\\
 g_5(\texttt{1})=\texttt{000010000011000111},\\
 g_5(\texttt{2})=\texttt{0000011}.\\
\end{array}$$

\begin{theorem}
Every binary bi-infinite word containing only the squares \textnormal{\texttt{00}}, \textnormal{\texttt{11}}, \textnormal{\texttt{0000}}, \textnormal{\texttt{0001100011}}, and \textnormal{\texttt{1000010000}}
has the same factor set as $g_5(b_3)$.
\end{theorem}

\begin{proof}
We say that a binary word, finite or infinite, is \emph{good} if it contains
no squares other than these five squares.
It is easy to check that a bi-infinite word is good if and only if it avoids $SQ_6$ and\\
$\{\texttt{101},\texttt{1001},\texttt{1111},\texttt{000000},\texttt{010001},
\texttt{100010},\texttt{0000010},\texttt{0100001},\texttt{1000110},\texttt{1110001}\}$.

First, we show that $g_5(b_3)$ is good.
A computer check shows that the factors of length 200 of $g_5(b_3)$ are good.
Now we have to show that $g_5(b_3)$ also avoids larger squares.
Unfortunately, checking directly that $g_5(b_3)$ avoids $SQ_6$ with Rosenfeld's program
is out of reach because 6 variables is too many for a morphism as large as $g_5$.
Thus, we had to resort to the following workaround.
Every factor of length 29 of $g_5(b_3)$ contains the factor \texttt{00000}.
So if $g_5(b_3)$ contains a square with period at least 31, then $g_5(b_3)$ contains an occurrence
of the pattern $P=ABBBBBCABBBBBC$ (such that $B\mapsto\texttt{0}$).
However, we have checked that $g_5(b_3)$ avoids~$P$. So $g_5(b_3)$ avoids $SQ_6$.

We construct the set $S_5^{100}$ defined as follows:
a word $v$ is in $S_5^{100}$ if and only if there exists a good word $pvs$ such that $|p|=|v|=|s|=100$.
We check that $S_5^{100}$ is exactly the set of factors of length 100 of $g_5(b_3)$.
%This implies that every bi-infinite good word $r_2$ has the same set of factors of length 100 as $g_5(b_3)$.
Using the same argument as in the previous proof, we obtain that every bi-infinite good word $r_2$
is in $\acc{g_5(\texttt{0}),g_5(\texttt{1}),g_5(\texttt{2})}^\omega$.
Then $r_2$ has the same factor set as $g_5(b_3)$ by~\Cref{lem:b3}.
\end{proof}

% echo ABAABA 3 012 02 1 0000100000111000011000111 000010000011000111 0000011 | ./dn
% The formula ABAABA is avoided by the morphism !!!!!!!!!!!!!!!!!!
% The formula AABBCAABBC is avoided by the morphism !!!!!!!!!!!!!!!!!!
% The formula ABBBBBCABBBBBC is avoided by the morphism !!!!!!!!!!!!!!!!!!

\subsection{Repetition threshold of walks on a graph}\label{ssec:p5}
A \emph{walk} in a vertex-labelled directed graph $G$ is a finite or infinite word
$w=w_1w_2\ldots$ such that every letter $w_i$ is a vertex and every factor $w_iw_{i+1}$ is an arc in $G$.
If, in full generality, we consider directed graphs such that 
loops and digons (two arcs with opposite direction) are allowed, then this is equivalent 
to consider words with allowed (or forbidden) factors of length 2.
Now we can consider the repetition threshold of $G$, that is, $RT(G)$ is the least critical exponent over infinite walks on $G$.
In most cases, there are exponentially many words that are $RT(G)^+$-free/$RT(G)$-free walks on $G$.

We revisit the results of Currie~\cite{Currie91} and Harju~\cite{harju2011} who have characterized
the (undirected) graphs that admit an infinite square-free walk.
If every component of a graph is a path on at most 4 vertices, then it has no infinite square-free walk.
Conversely, there exists an infinite square-free walk on the graphs $K_3$, $C_4$, $K_{1,3}$, and $P_5$, see~\Cref{fig:graphs}.
\begin{itemize}
 \item $RT(K_3)=\tfrac74^+$ by the result of Dejean~\cite{Dejean:1972}.
 \item $RT(C_4)=\tfrac53^+$ is a corollary of Lemma 16 in~\cite{badkobeh2022avoiding}.
 \item $RT(K_{1,3})=\tfrac{15}7^+$ is reached by the image of every ternary $\tfrac74^+$-free word by the morphism $\texttt{03}/\texttt{13}/\texttt{23}$.
\end{itemize}

The 5-path $P_5$ is noted $\texttt{201}\hat{\texttt{2}}\hat{\texttt{0}}$, see~\Cref{fig:graphs}. 
\begin{theorem}
Every bi-infinite square-free walk of $P_5$ has the same factor set as $k_5(b_3)$, where $k_5$ is\\
$$\begin{array}{rcl}
\texttt{0}&\to&{\texttt{01}}\hat{\texttt{2}}\hat{\texttt{0}}\hat{\texttt{2}}\texttt{1},\\
\texttt{1}&\to&{\texttt{01}}\hat{\texttt{2}}\texttt{1},\\
\texttt{2}&\to&\texttt{02}.
\end{array}$$
\end{theorem}

\begin{proof}
Clearly, $k_5(b_3)$ is a walk on $P_5$. To show that $k_5(b_3)$ is square-free,
we consider the coloring $c$ of the vertices such that $c(\texttt{0})=\texttt{0}$, $c(\texttt{1})=\texttt{1}$,
$c(\texttt{2})=\texttt{2}$, $c(\hat{\texttt 0})=\texttt{0}$, and $c({\hat{\texttt 2}})=\texttt{2}$.
Notice that $c\circ k_5=\texttt{012021}/\texttt{0121}/\texttt{02}$ is the square of $\texttt{012}/\texttt{02}/\texttt{1}$,
so that the colored walk is simply $b_3$. Since $c(k_5(b_3))$ is square-free, $k_5(b_3)$ is square-free.

Consider a bi-infinite square-free walk $r_5$ of $P_5$. Now, we check that \texttt{010} is not a factor of $r_5$.
Indeed, \texttt{010} extends to \texttt{20102}, then to \texttt{0201020}, and then to \texttt{102010201},
which contains the square $(\texttt{1020})^2$.
By symmetry, $\hat{\texttt{2}}{\texttt{1}}\hat{\texttt{2}}$ is not a factor of $r_5$ either.
Then we can check that the only factors of $r_5$ containing \texttt{0} only as a prefix and a suffix are
$\texttt{020}$, $\texttt{01}\hat{\texttt{2}}{\texttt 10}$,
and $\texttt{01}\hat{\texttt{2}}\hat{\texttt{0}}\hat{\texttt{2}}{\texttt 10}$.
Thus, $r_5$ is in $\{\texttt{01}\hat{\texttt 2}\hat{\texttt 0}{\hat{\texttt 2}}{\texttt 1},\texttt{01}{\hat{\texttt 2}}{\texttt 1},{\texttt{02}}\}^\omega$.

By~\Cref{lem:b3}, $r_5$ has the same factor set as $k_5(b_3)$.
\end{proof}

\begin{figure}[htbp]
\begin{center}
\includegraphics[width=12cm]{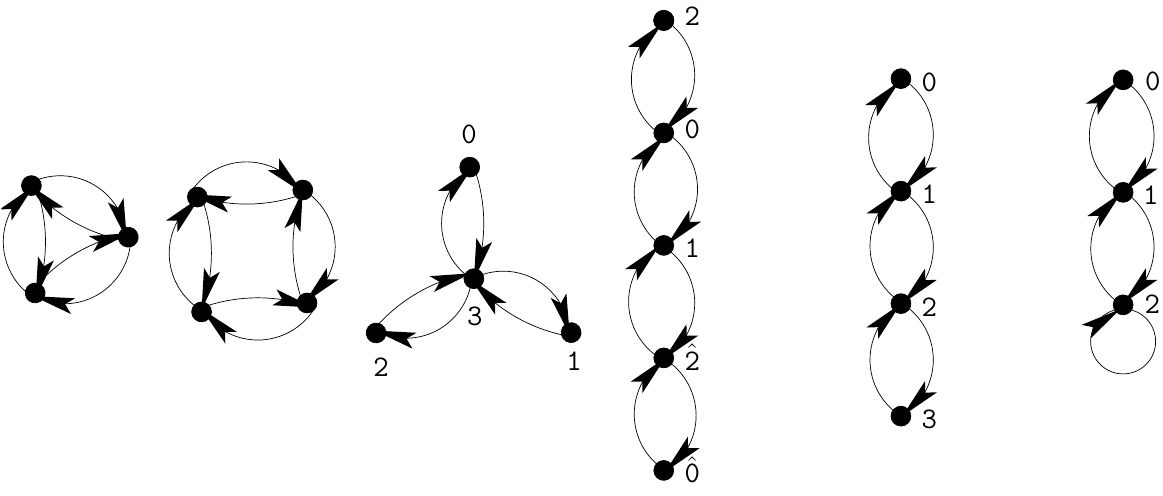}
\caption{The graphs $C_3$, $C_4$, $K_{1,3}$, $P_5$, $P_4$, and $P_3^\star$}
\label{fig:graphs}
\end{center}
\end{figure}

Beyond square-free walks, the graphs $P_4$ and $P_3^\star$ depicted in \Cref{fig:graphs} admit overlap-free walks, that is, $RT(P_4)=2^+$ and $RT(P_3^\star)=2^+$.
To prove this, we use an ad-hoc analog of~\Cref{lem:b3} suited for these overlap-free words.

\begin{lemma}\label{lem:b3+}
Let $w$ be a bi-infinite word and let $m:\Sigma^*_3\rightarrow\Sigma^*_k$ be such that 
\begin{itemize}
 \item $w$ is in $\acc{m(\textnormal{\texttt{0}}),m(\textnormal{\texttt{1}}),m(\textnormal{\texttt{2}})}^\omega$,
 \item $m(\textnormal{\texttt{0}})=abxb$, $m(\textnormal{\texttt{1}})=ab$, and $m(\textnormal{\texttt{2}})=ay$, where $a$ and $b$ are letters and $x$ and $y$ are possibly empty words.
 \item $w$ is overlap-free.
\end{itemize}
then $w$ has the same factor set as $m(b_3)$.
\end{lemma}

\begin{proof}
Since $w\in\acc{m(\texttt{0}),m(\texttt{1}),m(\texttt{2})}^\omega$, we write $w=m(r_3)$,
where $r_3$ is a bi-infinite ternary word.
If $r_3$ contains a square, then $m(r_3)$ contains an overlap because of the common prefix $a$.
So $r_3$ is square-free.
Moreover, $r_3$ avoids words of the form $\texttt{0}u\texttt{1}u\texttt{0}$, with $u\in\Sigma_3^*$, since
$m(\texttt{0}u\texttt{1}u\texttt{0})=abxbm(u)abm(u)abxb$ contains $bm(u)abm(u)ab$, which is an overlap.
Now, as proven in~\cite{aabbc}, a bi-infinite ternary square-free word avoids factors of the form
$\texttt{0}u\texttt{1}u\texttt{0}$ with $u\in\Sigma_3^*$ if and only if it avoids $\acc{\texttt{010},\texttt{212}}$.
So $r_3$ has the same factor set as $b_3$.
By taking the $m$-image, $w$ has the same factor set as $m(b_3)$.
\end{proof}

\begin{theorem}{\ }\label{thm:p4p3}
Let $k_4=\textnormal{\texttt{1232}}/\textnormal{\texttt{12}}/\textnormal{\texttt{10}}$ and $k_3=\textnormal{\texttt{122}}/\textnormal{\texttt{12}}/\textnormal{\texttt{10}}$.
\begin{itemize}
 %\item $2$ and $k_4(b_3)$ is characterized by infinite $2^+$-free walks on $P_4$.
 %\item $$ and $k_3(b_3)$ is characterized by infinite $2^+$-free walks on $P_3^\star$.
 \item Every bi-infinite overlap-free walk of $P_4$ has the same factor set as $k_4(b_3)$.
 \item Every bi-infinite overlap-free walk of $P_3^\star$ has the same factor set as $k_3(b_3)$.
\end{itemize}
\end{theorem}

\begin{proof}{\ }
\begin{itemize}
\item It is easy to check that $k_4(b_3)$ is a walk on $P_4$.
We have also checked that $k_4(b_3)$ avoids the pattern $ABABA$.
Since $k_4(b_3)$ contains no factor $aaa$ for some letter $a$, it is overlap-free.

Consider a bi-infinite overlap-free walk $r_4$ of $P_4$.
Notice that $r_4$ cannot contain \texttt{323} since it would extend to the overlap \texttt{23232}.
Then we can check that the only factors of $r_4$ containing \texttt{1} only as a prefix and a suffix are
\texttt{12321}, \texttt{121}, and \texttt{101}.
So $r_4$ is in $\acc{\texttt{1232},{\texttt{12}},{\texttt{10}}}^\omega$.
By~\Cref{lem:b3+}, $r_4$ has the same factor set as $k_4(b_3)$.

\item It is easy to check that $k_3(b_3)$ is a walk on $P_3^\star$.
We have also checked that $k_3(b_3)$ avoids the pattern $ABABA$.
Since $k_3(b_3)$ contains no factor $aaa$ for some letter $a$, it is overlap-free.

Consider a bi-infinite overlap-free walk $r_3$ of $P_3^\star$.
Since $r_3$ does not contain \texttt{222}, the only factors of $r_3$ containing \texttt{1}
only as a prefix and a suffix are \texttt{1221}, \texttt{121}, and \texttt{101}.
So $r_3$ is in $\acc{\texttt{122}, \texttt{12}, \texttt{10}}^\omega$.
By~\Cref{lem:b3+}, $r_3$ has the same factor set as $k_3(b_3)$.
\end{itemize}
\end{proof}

Notice that $k_3(b_3)$ is obtained from $k_4(b_3)$ by removing every letter \texttt{3}.

% ochem@ochem-2022:~/mots/ngh2$ echo ABABA 3 012 02 1 2101 21 23 | ./dn
% Size of border is 20
% The formula ABABA is avoided by the morphism !!!!!!!!!!!!!!!!!!
% The set of patterns is of size 721

% ochem@ochem-2022:~/mots/ngh2$ echo ABABA 3 012 02 1 100 10 12 | ./dn
% Size of border is 12
% The formula ABABA is avoided by the morphism !!!!!!!!!!!!!!!!!!
% The set of patterns is of size 721

\section{Images of $b_5$}\label{sec:b5}

\subsection{Twelve squares}\label{ssec:12sq}
The first author has shown (Proposition 3 in~\cite{Badkobeh-TCS}) that there exists a binary morphic image
of $b_3$ that contains 11 squares and the overlaps \texttt{10101}, \texttt{1001001}, and \texttt{0110110}.
This is best possible in the sense that we have checked by backtracking that no infinite binary word contains
at most 11 distinct squares and at most two overlaps in $\acc{\texttt{10101}, \texttt{1001001}, \texttt{0110110}}$.
The next result shows that we can get rid of \texttt{1001001} and \texttt{0110110} if we allow one more square.
Let $h_{12}$ be the morphism from $\Sigma_5^*$ to $\Sigma_2^*$ defined by

$$\begin{array}{l}
 h_{12}(\texttt{0})=\texttt{0011},\\
 h_{12}(\texttt{1})=\texttt{01},\\
 h_{12}(\texttt{2})=\texttt{001},\\
 h_{12}(\texttt{3})=\texttt{011},\\
 h_{12}(\texttt{4})=\varepsilon.\\
\end{array}$$

% \begin{theorem}
% Every binary recurrent word that contains only $\texttt{01010}$ as an overlap and only the 12 squares in
% $S_{12}=\{\texttt{00},
% \texttt{11},
% \texttt{0101},
% \texttt{1010},
% \texttt{010010},
% \texttt{01100110},
% \texttt{10011001},
% \texttt{100101100101},$\\$
% \texttt{011001011001},
% \texttt{100110100110},
% \texttt{101001101001},
% \texttt{10010110011010011001011001101001}\}$\\ has the same factor set as $h_{12}(b_5)$.
% \end{theorem}
% This set is symmetrical by reversal and unique up to complement
% (i.e. this set of 12 squares and its complement are the only ones that do the job).
\begin{theorem}
Every bi-infinite binary word containing only \textnormal{\texttt{01010}} as an overlap
and at most 12 distinct squares has the same factor set as $h_{12}(b_5)$.
\end{theorem}

\begin{proof}
We say that a binary word, finite or infinite, is \emph{good} if it contains no overlap
other than \texttt{01010} and no squares other than the following 12 squares:
$\texttt{0}^2$, $\texttt{1}^2$, $(\texttt{01})^2$, $(\texttt{10})^2$, $(\texttt{010})^2$,
$(\texttt{0110})^2$, $(\texttt{1001})^2$, $(\texttt{100110})^2$, $(\texttt{011001})^2$,
$(\texttt{101001})^2$, $(\texttt{100101})^2$, and $(\texttt{1001011001101001})^2$.
% \texttt{00}, \texttt{11}, \texttt{0101}, \texttt{1010}, \texttt{010010},
% \texttt{01100110}, \texttt{10011001}, \texttt{100110100110}, \texttt{011001011001},
% \texttt{101001101001}, \texttt{100101100101}, and \texttt{10010110011010011001011001101001}.
% A computer check shows that this property holds for factors of $h_{12}(b_5)$ of length 200.

First, we show that $h_{12}(b_5)$ is good.
A computer check shows that the factors of length 200 of $h_{12}(b_5)$ are good.
Now we have to show that $h_{12}(b_5)$ also avoids larger squares.
Unfortunately, checking directly that $h_{12}(b_5)$ avoids $SQ_{17}$ with Rosenfeld's program
is out of reach because 17 variables is too much. Thus, we had to resort to the following workaround.
Every factor of length 57 of $h_{12}(b_5)$ contains the factor \texttt{0110010100110}.
So if $h_{12}(b_5)$ contains a square with period at least 59, then $h_{12}(b_5)$ contains an occurrence
of the pattern $P=ABCCBBCBCBBCCBDABCCBBCBCBBCCBD$ (such that $B\mapsto\texttt{0}$ and $C\mapsto\texttt{1}$).
However, we have checked that $h_{12}(b_5)$ avoids~$P$. So $h_{12}(b_5)$ avoids $SQ_{17}$.

Now we prove the characterization. Consider Thue's word $w_3=M_2(b_5)$ where
$M_2=\texttt{02}/\texttt{1}/\texttt{0}/\texttt{12}/\varepsilon$.
The word $w_3$ avoids (and is characterized by) $AA$ and $\acc{\texttt{121},\texttt{212}}$~\cite{BO15}.
See also Theorem 15 in~\cite{OchemRosenfeld2021} for another characterization.
We notice that $h_{12}=g_{12}\circ M_2$, where $g_{12}=\texttt{001}/\texttt{01}/\texttt{1}$.
This implies that $h_{12}(b_5)=g_{12}(M_2(b_5))=g_{12}(w_3)$.

We construct the set $S_{12}^{120}$ defined as follows:
a word $v$ is in $S_{12}^{120}$ if and only if there exists a good word $pvs$ such that $|p|=|v|=|s|=120$.
We check that $S_{12}^{120}$ is exactly the set of factors of length 120 of $g_{12}(w_3)$.
Consider a bi-infinite good word $r_2$. Since $r_2$ avoids \texttt{000}, $r_2$ is in $\acc{\texttt{001},\texttt{01},\texttt{1}}^\omega$,
that is, $r_2=g_{12}(r_3)$ for some bi-infinite ternary word $r_3$.
% Now, $g_{12}(r_3)$ and $g_{12}(w_3)$ have the same set of factors of length 120.
% So $r_3$ and $w_3$ have the same set of factors of length $\tfrac{120}{3}=40$.
Since $g_{12}$ is a code and the factors of length 120 of $g_{12}(r_3)$ are factors of $g_{12}(w_3)$,
then the factors of length $\tfrac{120}{3}=40$ of $r_3$ are factors of $w_3$.
In particular, $r_3$ avoids \texttt{121}, \texttt{212} and squares with period at most 20.
Then $r_3$ also avoids squares with period at least 21, since otherwise $g_{12}(r_3)$
would contain a square with period at least 21, contradicting that $r_2=g_{12}(r_3)$ is good.
So $r_3$ avoids $AA$ and $\acc{\texttt{121},\texttt{212}}$, which implies that 
$r_3$ has the same factor set as $w_3$.
By taking the $g_{12}$-image, $r_2$ has the same factor set as $g_{12}(w_3)$.
\end{proof}

% g(121)=01101 and 101101 is forbidden
% g(212)=1011 and 11011 is forbidden
% echo AABCCAABCC 5 01 23 4 21 0 001101001011 00101 0011 001011 001101 | ./dn
% The formula AABCCAABCC is avoided by the morphism !!!!!!!!!!!!!!!!!!
% The formula AABBCAABBC is avoided by the morphism !!!!!!!!!!!!!!!!!!
% The formula ABABBACABABBAC is avoided by the morphism !!!!!!!!!!!!!!!!!!=)
% The formula AABACAABAC is avoided by the morphism !!!!!!!!!!!!!!!!!!
% The formula ABCBBABCBB is avoided by the morphism !!!!!!!!!!!!!!!!!!
% The formula ABAABCABAABC is avoided by the morphism !!!!!!!!!!!!!!!!!!
% The formula ABCCBCABCCBC is avoided by the morphism !!!!!!!!!!!!!!!!!!
% The formula ABBAABCABBAABC is avoided by the morphism !!!!!!!!!!!!!!!!!!
% The formula ABBABABCABBABABC is avoided by the morphism !!!!!!!!!!!!!!!!!!
% The formula ABBACBAABABBACBAAB is avoided by the morphism !!!!!!!!!!!!!!!!!!
% The formula ABBACABABDABBACABABD  is NOT avoided!!
% 1440000 done out of 1447043
% The formula ABCCBBCBCBBCCBDABCCBBCBCBBCCBD is avoided by the morphism !!!!!!!!!!!!!!!!!!

% It avoids $SQ_{17}$ and $\{\texttt{000}, \texttt{111}, \texttt{10101}, \texttt{00100},
% \texttt{11011}, \texttt{101101}, \texttt{01001010},$\\
% $\texttt{01010010}, \texttt{00110011}, \texttt{11001100}, \texttt{1011001011}, \texttt{1101001101}, \texttt{010011010011},$\\
% $\texttt{110010110010}, \texttt{01010011001010}, \texttt{010010110011010010}, \texttt{110011010010110011},$\\
% $\texttt{110010110011010011001011}, \texttt{110100110010110011010011}\}$.\\

\subsection{Few occurrences of $ABBA$}\label{ssec:abba}

A pattern $P$ is \emph{doubled} if every variable appears at least twice in $P$, i.e., $P$ contains no isolated variable.
It is known that doubled patterns are 3-avoidable~\cite{O16}, that many doubled patterns are 2-avoidable~\cite{DO2023},
and it is conjectured that only finitely many doubled patterns are not 2-avoidable.
Let $oc_k(P)$ be the minimum number of distinct occurrences of $P$ in an infinite word over $\Sigma_k$.
So $oc_k(P)=0$ if $P$ is $k$-avoidable and we leave as an exercise the proof that $oc_k(P)=\infty$
if $P$ is not doubled and is not $k$-avoidable.
The many known constructions of infinite binary words with only three distinct squares (see~\cite{GS2021} for a survey)
imply that $oc_2(AA)=3$ and that there are exponentially many such words.
We have checked\footnote{We would like to thank Antoine Domenech for his help in the study of $oc_k(P)$}
that for a few other 2-unavoidable doubled patterns $P$,
there are also exponentially many binary words containing at most $oc_2(P)$ distinct occurrences of $P$.
% Moreover there are With the help of , our investigation suggests that for most 
% However, the following result shows that this is not the case for $oc_2(ABBA)=8$.

However, this is not the case for the pattern $ABBA$.
Let us say that a binary word, finite or not, is \emph{thrifty} if it contains at most 8 distinct occurrences of $ABBA$.
Also, let $c$ be the morphism from $\Sigma_5^*$ to $\Sigma_2^*$ defined by
$$\begin{array}{l}
 c(\texttt{0})=\texttt{0010111100},\\
 c(\texttt{1})=\texttt{1101000011},\\
 c(\texttt{2})=\varepsilon,\\
 c(\texttt{3})=\texttt{1101001100},\\
 c(\texttt{4})=\texttt{0010110011}.\\
\end{array}$$
Then the following result implies that $oc_2(ABBA)=8$.
\begin{theorem}
Every bi-infinite thrifty word has the same factor set as $c(b_5)$.
\end{theorem}
\begin{proof}
Let $X_8=\acc{\textnormal{\texttt{0000}}, \textnormal{\texttt{0110}}, \textnormal{\texttt{1001}}, \textnormal{\texttt{1111}},
\textnormal{\texttt{001100}}, \textnormal{\texttt{011110}}, \textnormal{\texttt{100001}}, \textnormal{\texttt{110011}}}$ and
$F=\{\texttt{00000},\texttt{01010},\texttt{01110},\texttt{0001101},\texttt{0100111},\texttt{01000010},\texttt{00111100},\texttt{11111},$\\$\texttt{10101},\texttt{10001},\texttt{1110010},\texttt{1011000},\texttt{10111101},\texttt{11000011}\}$.
% We have checked by computer that:
% \begin{enumerate}[(i)]
% \item $c(b_5)$ avoids $F$ and $SQ_3$.\label{i}
% \item $X_8$ contains every occurrence $h$ of $ABBA$ in $c(b_5)$ such that $|h(AB)|\le100$.\label{ii}
% \item No infinite binary word contains at most 7 distinct occurrences of $ABBA$, that is, $oc_2(ABBA)\ge8$.\label{iii}
% \item No infinite binary word contains at most 8 distinct occurrences of $ABBA$ and avoids the prefix of length 200 of $c(b_5)$.\label{iv}
% \item \label{v}
% \end{enumerate}
It is easy to see that $c(b_5)$ avoids $F$ and we have checked that $c(b_5)$ avoids $SQ_3$.
Then we use the technique in~\cite{BO15} to show that every bi-infinite binary word
avoiding $SQ_3\cup F$ has the same factor set as $c(b_5)$. So $SQ_3\cup F$ characterizes $c(b_5)$.

Now, we show by contradiction that $c(b_5)$ does not contain an occurrence of $ABBA$ outside of $X_8$.
A computer check shows that $X_8$ contains every occurrence $h$ of $ABBA$ in $c(b_5)$ such that $|h(AB)|\le100$.
So we can assume that $|h(AB)|\ge101$.
Because of the square $BB$ and since $c(b_5)$ avoids $SQ_3$, we have $1\le|h(B)|\le2$, so that $|h(A)|\ge99$.
So it is sufficient to prove that $c(b_5)$ avoids $ABCDDABC$.
Unfortunately, a direct computation was not finished after several days and we have to use the following case analysis.
Notice that all the factors \texttt{0101} and \texttt{1010} start at position $1\pmod{10}$ in $c(b_5)$. 
Moreover, $h(ABBA)$ is a repetition of period $h(ABB)$.
This implies that $|h(ABB)|$ is a multiple of $10$.
\begin{itemize}
\item $h(B)\in\acc{\texttt{00},\texttt{11}}$: This is ruled out because $c(b_5)$ avoids the pattern $ABCCCCAB$.
\item $h(B)\in\acc{\texttt{01},\texttt{10}}$: If $c(b_5)$ contains $h(A)\texttt{0101}h(A)$, then the letter \texttt{0}
is a suffix of the left occurrence of $h(A)$ in position $0\pmod{10}$. Since $|h(ABB)|$ is a multiple of $10$,
then \texttt{0} is a suffix of $h(A)\texttt{0101}h(A)$ in position $0\pmod{10}$.
So $h(A)\texttt{0101}h(A)$ extends to $h(A)\texttt{0101}h(A)\texttt{0101}$, which is ruled out since $c(b_5)$ avoids $SQ_3$. 
The case $h(B)=\texttt{10}$ is ruled out similarly.
\item $h(B)\in\acc{\texttt{0},\texttt{1}}$: Suppose that $c(b_5)$ contains $h(A)\texttt{00}h(A)$.
If $h(A)=\texttt{0}x$, then $h(A)\texttt{00}h(A)=\texttt{0}x\texttt{000}x$ is an occurrence of $ABAAAB$.
However, $c(b_5)$ avoids $ABAAAB$.
Similarly, if $h(A)=x\texttt{0}$, then $h(A)\texttt{00}h(A)=x\texttt{000}x\texttt{0}$ is an occurrence of $ABBBAB$ and $c(b_5)$ avoids $ABBBAB$.
Therefore $h(A)=\texttt{1}x\texttt{1}$. So $h(A)\texttt{00}h(A)=\texttt{1}x\texttt{1001}x\texttt{1}$
is an occurrence of $ABACCABA$. However, $c(b_5)$ avoids $ABACCABA$.
The case $h(B)=\texttt{1}$ is ruled out similarly.
\end{itemize}
So $c(b_5)$ is thrifty and $X_8$ is exactly the set of occurrences of $ABBA$ in $c(b_5)$.

Finally, we show that every thrifty bi-infinite word $w_8$ has the same factor set as $c(b_5)$, that is, avoids $SQ_3\cup F$.
We construct the set $S_8^{100}$ such that
a word $v$ is in $S_8^{100}$ if and only if there exists a thrifty word $pvs$ where $|p|=|v|=|s|=100$.
We check that $S_8^{100}$ is exactly the set of factors of length 100 of $c(b_5)$.
% Then we consider any binary bi-infinite word $w_8$ containing at most 8 distinct occurrences of $ABBA$ contains every factor in $X_8$.
% We compute the set $K^{100}$ of words of length 100 that appear as central
% factor of a binary word of length $120$ (?) containing only the elements of $X_8$ as an occurrence of $ABBA$.
% We verify that $K^{100}$ is equal to the factor set of $c(b_5)$ of length $100$.
This implies that $w_8$ avoids $F$ and squares $uu$ with $3\le|u|\le50$.
So there remains to show that $w_8$ avoids $SQ_{51}$.
Suppose for contradiction that $w_8$ contains a square $uu$ with $|u|\ge 7$. We write $u=xvy$, where $|x|=|y|=2$ and thus $|v|\ge3$.
The factor $yx$ of $uu=xvyxvy$ has length 4, so it contains a square $cc$ and we write 
$yx=pccs$, where $p$ is a possibly empty prefix of $y$ and $s$ is a possibly empty suffix of $x$.
Now, $uu=xvpccsvy$ contains the factor $svpccsvp$.
This is an occurrence of $ABBA$ such that $|h(A)|=|svp|\ge|v|\ge3$, thus it is not in $X_8$.
So $w_8$ avoids $SQ_7$, then $w_8$ avoids $SQ_3\cup F$ and finally $w_8$ has the same factors as $c(b_5)$.

\end{proof}

\section{The period-doubling sequence}\label{sec:pd}
Recall that the period-doubling word ${\bf pd}$ is the fixed point of the morphism $z=\texttt{01}/\texttt{00}$.
Recently, James Currie~\cite{Currie2023} has obtained a characterization of ${\bf pd}$.
Our last result is another characterization of ${\bf pd}$.
We first reprove Currie's result for completeness, and also to reuse parts of the proof in order
to obtain our characterization.
\begin{theorem}~\cite{Currie2023}
Every bi-infinite binary word avoiding $\acc{AAAA, AAABABAA,\texttt{11},\texttt{1001}}$ has the same factor set as ${\bf pd}$.
\end{theorem}
\begin{proof}
Let $w$ be a bi-infinite binary word avoiding $\acc{AAAA, AAABABAA,\texttt{11},\texttt{1001}}$.
Notice that $w$ avoids \texttt{0000} as it is an occurrence of $AAAA$.
Let $\texttt{10}^k\texttt{1}$ be a factor of $w$. Then the forbidden factors \texttt{11},\texttt{1001}, and
\texttt{0000} show respectively that $k\ne0$, $k\ne2$, and $k<4$.
Thus $k\in\acc{1,3}$, so that $w\in{}^\omega\hspace{-1,5mm}\acc{\texttt{01},\texttt{0001}}^\omega$.
Therefore, $w\in{}^\omega\hspace{-1,5mm}\acc{\texttt{01},\texttt{00}}^\omega$
and there exists a bi-infinite binary word $v$ such that $w=z(v)$.
Then $v$ must avoid $AAAA$ and $AAABABAA$. Moreover, $v$ avoids \texttt{11} because
$z(\texttt{11})=\texttt{0000}$ is an occurrence of $AAAA$ and $v$ avoids \texttt{1001} because
$z(\texttt{1001})=\texttt{00010100}$ is an occurrence of $AAABABAA$.
So $v$ avoids $\acc{AAAA, AAABABAA,\texttt{11},\texttt{1001}}$ too, which finishes the proof.
\end{proof}

\begin{theorem}
Every bi-infinite binary word avoiding $\acc{AA.ABAB.BB,\texttt{11}}$ has the same factor set as ${\bf pd}$.
\end{theorem}
\begin{proof}
Let $w$ be a bi-infinite binary word avoiding $\acc{AA.ABAB.BB,\texttt{11}}$.
Backtracking shows that no infinite binary word avoids $\acc{AA.ABAB.BB,\texttt{11},\texttt{1010}}$.
Thus $w$ contains \texttt{1010}.
Suppose for contradiction that $w$ contains \texttt{1001}.
Since $w$ avoids \texttt{11}, then $w$ contains \texttt{010010}.
Now $w$ contains \texttt{1010} and \texttt{010010}, so that $w$ contains the occurrence $A\mapsto\texttt{0}$;
$B\mapsto\texttt{10}$ of $AA.ABAB.BB$.
This contradiction means that $w$ avoids \texttt{1001}.
Notice that $w$ also avoids \texttt{0000} as it is an occurrence of $AA.ABAB.BB$.
As in the previous proof, $w$ avoids \texttt{11},\texttt{1001}, and \texttt{0000},
so that there exists a bi-infinite binary word $v$ such that $w=z(v)$.
Then $v$ must avoid $AA.ABAB.BB$. Moreover, $v$ avoids \texttt{11} because
$z(\texttt{11})=\texttt{0000}$ is avoided by $w$.
So $v$ avoids  $\acc{AA.ABAB.BB,\texttt{11}}$ too, which finishes the proof.
\end{proof}

\end{document}